\documentclass[10pt,twoside]{amsart}

\usepackage{color}
\usepackage{amsthm}
\usepackage{amssymb}
\usepackage{amsmath}
\usepackage{mathrsfs}
\usepackage{mathtools}
\usepackage{graphicx}
\usepackage{bm}
\usepackage{dsfont}
\usepackage{verbatimbox}
\usepackage{mathdots}
\usepackage{enumitem}
\usepackage{booktabs}
\usepackage[labelsep=colon,labelfont={color=ver}]{caption}
\usepackage{hyperref}
\usepackage{comment}

\theoremstyle{plain}
\newtheorem{thm}{Theorem}

\newtheorem{prop}[thm]{Proposition}

\theoremstyle{definition}
\newtheorem{defn}{Definition}

\DeclareMathOperator{\wt}{\mathbf{wt}}

\definecolor{referencias}{rgb}{0.1,0.5,0.2}
\definecolor{citation}{rgb}{.9,0.1,0}
\definecolor{ver}{rgb}{0.0, 0.42, 0.24}
\definecolor{applegreen}{rgb}{0.55, 0.71, 0.0}

\definecolor{azul}{rgb}{.204,.353,.541}
\definecolor{azulclaro}{rgb}{.31,.506,.741}
\definecolor{verde}{rgb}{0.1,0.5,0.2}
\definecolor{rojo}{rgb}{.9,0.1,0}
\definecolor{lightgray}{gray}{0.8}
\definecolor{gris}{gray}{0.3}
\definecolor{grissuave}{gray}{0.5}

\definecolor{light-gray}{gray}{0.92}
\definecolor{MyGray}{rgb}{0.90,0.90,0.90}
\definecolor{lgray}{gray}{0.85}
\definecolor{mgray}{gray}{0.7}
\definecolor{myred}{RGB}{151,20,20}
\definecolor{myblue}{rgb}{0.2,0.2,0.7}



\def    \N      {\mathds{N}}

\newcommand{\dsum}{\displaystyle\sum}

\newcommand{\ra}[1]{\renewcommand{\arraystretch}{#1}}
\newcommand{\tifrac}[2]{\mbox{\tiny$\frac{#1}{#2}$}}

\title{A Lucas analogue of Eulerian numbers}
\author[J.\ Agapito Ruiz]{Jos\'e Agapito Ruiz}
\address{\parbox{\linewidth}{Centro de an\'alise funcional, estruturas lineares e aplica\c{c}\~oes, faculdade de ci\^encias, \\ universidade de lisboa, 1749-016 Lisboa, portugal\\[1ex]}}
\email{jaruiz@ciencias.ulisboa.pt}
\thanks{2020 \emph{Mathematics Subject Classification.}
Primary 03, 05, 11}

\begin{document}

\begin{abstract} The generalized Lucas numbers are polynomials in two variables with nonnegative integer coefficients. Lucas versions of some combinatorial numbers with known formulas in terms of quotient and products of nonnegative integers have been recently given by replacing the integers in those formulas with their corresponding Lucas analogues. We instead use a recursive approach. In this sense, we give a recursive formula for Lucas-Narayana numbers derived from a recent formula in terms of Lucasnomials (the explicit Lucas version of binomial numbers). We propose a recursive definition for a Lucas analogue of the classical Eulerian numbers, which shows immediately that they are polynomials in two variables with nonnegative integer coefficients. We prove that they are palindromic like their standard counterparts. The recursive approach allows  us to give Lucas analogues of many relevant combinatorial constants. In particular, Lucas versions for both Stirling numbers of the second kind and Motzkin numbers are presented. 
\end{abstract}

\maketitle

\section{Introduction}
\label{se:1} 

The generalized Lucas \emph{numbers} are defined recursively by $\{0\}=0$, $\{1\}=1$, and 
\begin{equation}\label{eq:Lucas-seq}
\{n\} = s\{n-1\}+t\{n-2\}\,,
\end{equation}

\noindent for any integer $n\ge 2$. It follows immediately from \eqref{eq:Lucas-seq} that the generalized Lucas numbers $\{n\}$ are polynomials in $s$ and $t$ with nonnegative integer coefficients. On specialization of $s$ and $t$ one can recover, for example, the nonnegative integers ($s=2$, $t=-1$), the $q$-integers $[n]_q=\frac{q^n-1}{q-1}=1+q+\cdots+q^{n-1}$ ($s=1+q$, $t=-q$), the Chebyshev polynomials of the second kind $U_{n-1}(x)$ ($U_0(x)=1$, $U_1(x)=2x$, $s=2x$, $t=-1$), and the Fibonacci sequence $F_n$ ($s=t=1$), which is at the core of recursion~\eqref{eq:Lucas-seq}. Here are a few instances of the generalized Lucas numbers:
\begin{equation*}
\{2\} = s,\quad \{3\}=s^2+t,\quad \{4\}=s^3+2st,\quad \{5\}=s^4+3s^2t+t^2,\quad \{6\}= s^5+4s^3t+3st^2\,.
\end{equation*}

A useful combinatorial interpretation of $\{n\}$ is derived from the standard interpretation of Fibonacci numbers  via tiling. In this regard, we consider two types of tiles: monominoes (which cover one unit squares) and dominoes (which cover two unit squares), so that for a given configuration $T$ using these tiles, we define its \emph{weight} to be 
\begin{equation}\label{eq:weight_T}
\wt T=s^{\text{$\#$ of monominoes in }T}\,\, t^{\text{$\#$ of dominoes in }T}\,.
\end{equation}
 Furhermore, for any set $\mathscr{T}$ of tilings of this kind, the \emph{weight} of $\mathscr{T}$ is defined to be 
\begin{equation}\label{eq:weightTcoll}
\wt\mathscr{T} = \dsum_{T\in\mathscr{T}} \wt T\,.
\end{equation}
Clearly, $\wt\mathscr{T}$ is a polynomial in $s$ and $t$ when $\mathscr{T}$ is finite, otherwise it is a formal power series. 

\begin{figure}[ht]
\centering
\includegraphics[scale=0.8]{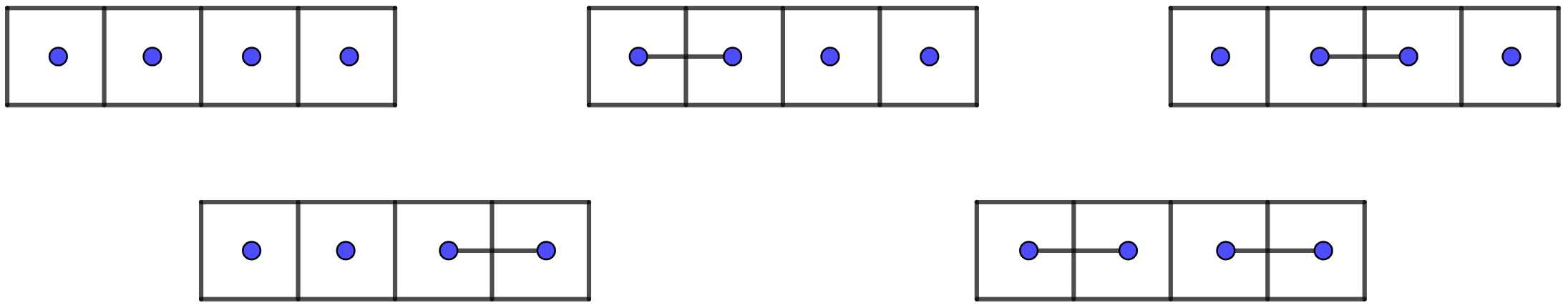}
\caption{The tilings in $\mathscr{T}(4)$}
\label{fi:tiling}
\end{figure}

Let $\mathscr{T}(n)$ be the set of all tilings of a row of $n$ squares. Figure~\ref{fi:tiling} shows all possible tilings for $n=4$. It is then easy to see by induction that for every $n\ge 1$, we have
\begin{equation*}
 \{n\} = \wt\mathscr{T}(n-1)\,.
 \end{equation*}
 \noindent The cardinality of $\mathscr{T}(n-1)$ is precisely the Fibonacci number $F_n$ (clearly, for $\mathscr{T}(0)=\emptyset$ we assume $\vert\mathscr{T}(0)\vert=1$). Now, by tiling rows with a decreasing number of squares and piling them up we can make sense of the expression 
\begin{equation}\label{eq:Lucastorial}
\{n\}!=\{n\}\{n-1\}\cdots\{2\}\{1\}\,.
\end{equation}
 
\noindent This is the Lucas analogue of the factorial, the \emph{Lucastorial} number. By convention $\{0\}!=1$. Using the setting of Young diagrams, Identity~\eqref{eq:Lucastorial} gives the weight of the set of all tilings of the standard staircase partition $\delta_n=(n-1,n-2,\ldots,1)$; namely 
\begin{equation*}
\wt\mathscr{T}(\delta_n)=\{n\}!
\end{equation*}

\noindent Likewise, one can obtain a Lucas version of any expression given in terms of products and quotients of nonnegative integers by simply replacing each ocurrence of $n$ in the expression with $\{n\}$. For instance, given $0\le k\le n$, the \emph{Lucasnomial} is defined by
 \begin{equation}\label{eq:Lucasnomials}
 \begin{Bmatrix}
 n \\
 k
 \end{Bmatrix} = \frac{\{n\}!}{\{k\}!\{n-k\}!}\,.
 \end{equation}
 
\noindent From this point of view, Lucasnomials are rational expressions in $s$ and $t$ a priori. However they are indeed  polynomials in $s$ and $t$ with nonnegative integer coefficients. This follows from some combinatorial interpretations given by Gessel and Viennot~\cite{GV_85}, Benjamin and Plott~\cite{BP_08}, and Sagan and Savage \cite{SaSa_10}). More recently, using a model that involves  lattice paths inside tilings of Young diagrams, Bennett et al.\cite{Bennett_et_al_20} have given another combinatorial interpretation for Lucasnomials that not only is simpler than the previous ones but also is more flexible for combinatorial proofs of identities involving Lucasnomials. Their combinatorial interpretation is also extendable to other familiar sequences of numbers like the Catalan numbers and their relatives. It is worth mentioning that Amdeberhan et.~al.~\cite{Amdeberhan_et_al_14} have also studied the generalized Lucas sequence and the Lucas analogues of Binomial and Catalan numbers from an algebraic and number theoretical perspective. 

In general, given any array of numbers $A_{n,k}$, its Lucas analogue will be denoted by $A_{\{n,k\}}$. An alternative natural way to obtain Lucas analogues of standard arrays of numbers is by generalizing their recursive relations in a straightforward manner. For instance, it is well-known that the following formula holds for binomial numbers $B_{n,k}=\binom{n}{k}$,
\begin{equation}\label{eq:Binom-rec1}
B_{n,k}=B_{n-1,k}+B_{n-1,k-1}\,.
\end{equation}
However, since $\{1\}=1$, we do not get anything new by replacing $1$ with $\{1\}$ in \eqref{eq:Binom-rec1}. Therefore, we have to choose a different recursive relation. In this regard, there is a less familiar recursive formula for binomial numbers that we can work with; that is 
\begin{equation}\label{eq:Binom-rec2}
B_{n,k}=(k+1)B_{n-1,k}-(n-k-1)B_{n-1,k-1}\,.
\end{equation}
Keeping in mind that $\{n\}=n$ when $s=2$ and $t=-1$, a natural way of defining the Lucas analogue of \eqref{eq:Binom-rec2} is
\begin{equation}\label{eq:Lucas-Binom-rec}
B_{\{n,k\}} = \{k+1\} B_{\{n-1,k\}} + t\{n-k-1\} B_{\{n-1,k-1\}}\,.
\end{equation}
An immediate conclusion from \eqref{eq:Lucas-Binom-rec} is that $B_{\{n,k\}}$
is a polynomial in $s$ and $t$ with nonnegative integer coefficients. Moreover, using as initial conditions $B_{\{n,0\}}=B_{\{n,n\}}=1$ for $n\ge 0$ and $B_{\{n,k\}}=0$ whenever $n<k$ or $k<0$, we can recover by induction the formula 
\begin{equation*}
B_{\{n,k\}} = \frac{\{n\}!}{\{k\}!\{n-k\}!} \,.
\end{equation*}
Hence, the Lucas version for binomial numbers obtained from generalizing the recursive relation~\eqref{eq:Binom-rec2} agrees with the explicit formula for Lucasnomials.  

Another familiar array of nonnegative integers with combinatorial interpretations is the array of Narayana numbers, defined for $1\le k\le n$ as
\begin{equation}\label{eq:Na}
N_{n,k} = \frac{1}{n}\binom{n}{k}\binom{n}{k-1}\,.
\end{equation}
The corresponding Lucas-Narayana numbers are then given by 
\begin{equation}\label{eq:LuNa}
N_{\{n,k\}} =\frac{1}{\{n\}}\begin{Bmatrix}n \\ k\end{Bmatrix}\begin{Bmatrix}n \\ k-1\end{Bmatrix}\,.
\end{equation}

\noindent The fact that Lucas-Narayana numbers are indeed polynomials in $s$ and $t$ with nonnegative integer coefficients was conjectured in~\cite{Bennett_et_al_20}, and later proved by Sagan and Tirrell~\cite{ST_20} using special factoring properties of $\{n\}$. Recently, Garret and Killpatrick~\cite{GK_22} 
gave a combinatorial proof of this fact by showing that the Lucas-Narayana numbers $N_{\{n,k\}}$ satisfy the formula
\begin{equation}\label{eq:LuNa-bin}
N_{\{n,k\}} = \begin{Bmatrix}n-1 \\ k-1\end{Bmatrix}^2+t\begin{Bmatrix}n-1\\ k\end{Bmatrix}\begin{Bmatrix}n-1 \\ k-2\end{Bmatrix}\,,
\end{equation}

\noindent for $2\le k\le n-1$ and $n\ge 1$. Formula~\eqref{eq:LuNa-bin} shows that $N_{\{n,k\}}$ is a polynomial in $s$ and $t$ with nonnegative integer coefficients because it is expressed as a sum of products of Lucasnomials with coefficients in $\N[s,t]$. Now, manipulating Formula~\eqref{eq:LuNa-bin} algebraically,  we get a recursive formula for Lucas-Narayana numbers. 

\begin{prop}
\label{pr:1}
For any integers $2\le k\le n-1$ and $n\ge 1$, the Lucas-Narayana numbers are given by
\begin{equation}\label{eq:LuNa-rec}
N_{\{n,k\}} = \frac{\{k\}\{n-1\}}{\{n-k\}} \, N_{\{n-1,k\}} + t\,\frac{\{n-k\}\{n-1\}}{\{k\}}  N_{\{n-1,k-1\}}\,,
\end{equation}
with initial conditions $N_{\{0,0\}}=1$, $N_{\{n,k\}}=1$ for $n=k$ or $k=1$, $N_{\{j,0\}}=0$ for $j>0$, and $N_{\{n,k\}}=0$ for $n<k$.
\end{prop}

\noindent Formula \eqref{eq:LuNa-rec} is the Lucas analogue of the not well-known recurrence relation satisfied by the classical Narayana numbers,
\begin{equation*}\label{eq:Na-rec2}
N_{n,k} = \frac{k(n-1)}{n-k}N_{n-1,k} - \frac{(n-k)(n-1)}{k} N_{n-1,k-1}\quad .
\end{equation*}

\noindent We can recover from Proposition~\ref{pr:1} the closed formula~\eqref{eq:LuNa}, by using induction again. Note that the polynomial nature of $N_{\{n,k\}}$ with nonnegative integer coefficients is not immediate from the recursive formula~\eqref{eq:LuNa-rec} this time.

These two examples illustrate the elementary approach we will take for the rest of the paper. We show in Section~\ref{se:2} that a straightforward Lucas version of a recursive relation satisfied by the classical Eulerian numbers $E_{n,k}$ gives a way of defining the Lucas-Eulerian numbers $E_{\{n,k\}}$. It follows immediately from the definition that they are polynomials in $s$ and $t$ with nonnegative integer coefficients.  We show that they are palindromic. Using the recursive definition of $E_{\{n,k\}}$, we get a formula for $E_{\{n,1\}}$ that encourage us to conjecture a similar formula for the general case $E_{\{n,k\}}$. In Section~\ref{se:3}, we elaborate on the difficulties that arise when trying to give a Lucas analogue of a known alternating sum formula for Eulerian numbers that be compatible with our Lucas-Eulerian numbers. We end in  Section~\ref{se:4} by making further comments on open problems and on Lucas analogues of Stirling numbers of the second kind and Motzkin numbers.

\section{Lucas-Eulerian numbers}
\label{se:2}

Among several equivalent variants of a known recurrence relation satisfied by the classical Eulerian numbers, we have the following formula
\begin{equation}\label{eq:Eul-rec}
E_{n,k} = (k+1) E_{n-1,k} + (n-k+1) E_{n-1,k-1}\,.
\end{equation}

\noindent A straightforward Lucas analogue of Formula~\ref{eq:Eul-rec} defines the Lucas-Eulerian numbers. 

\begin{defn} \label{de:LuEul}The Lucas-Eulerian numbers are defined recursively by the formula
\begin{equation}\label{eq:LuEul-rec}
E_{\{n,k\}} = \{k+1\}E_{\{n-1,k\}} + \{n-k+1\}E_{\{n-1,k-1\}}\,,
\end{equation}
for any integers $0 \le k \le n$, with initial conditions  $E_{\{n,0\}}=E_{\{n,n\}}=1$ for $n\ge 0$ and $E_{\{n,k\}}=0$ for $n<k$ or $k<0$. 
\end{defn}

\begin{table}[h]
\tiny
\begin{minipage}{\textwidth}
\centering
\ra{1.2}
\begin{tabular}{c c c c c c p{0.8cm} p{0.8cm}}
\toprule
 $\bm{n\backslash k}$ & ${\color{red}0}$ & ${\color{red}1}$ & ${\color{red}2}$ & ${\color{red}3}$ & ${\color{red}4}$ & ${\color{red}5}$  \\
\midrule
\color{red}0 & 1 & 0 & 0 & 0 & 0 & 0 \\
\color{red} 1 & 1 & 1 & 0 & 0 & 0 & 0 \\
\color{red}2 & 1 & $2s$ & 1 & 0 & 0 & 0 \\
\color{red}3 & 1 & $3s^2+t$ & $3s^2+t$ & 1 & 0 & 0 \\
\color{red}4 & 1 & $4s^3+3st$ & $6s^4+8s^2t+2t^2$ & $4s^3+3st$ & 1 & 0 \\
\color{red}5 & 1 & $5s^4+6s^2t+t^2$ & $10s^6+25s^4t+16s^2t^2+2t^3$ & $10s^6+25s^4t+16s^2t^2+2t^3$ & $5s^4+6s^2t+t^2$ & 1 \\
\bottomrule
\end{tabular}
\caption{A sample of Lucas-Eulerian numbers}
\label{table:1}
\end{minipage}
\end{table}

\noindent It immediately follows from Definition~\ref{de:LuEul} that the Lucas-Eulerian numbers are polynomials in $s$ and $t$ with nonnegative integer coefficients. Moreover, Table~\ref{table:1} indicates that they are palindromic. Note that for Lucasnomials $B_{\{n,k\}}$ and Lucas-Narayana numbers $N_{\{n,k\}}$, this property is immediate because they both have closed formulas from where it is easy to check that $B_{\{n,k\}}=B_{\{n,n-k\}}$ and $N_{\{n,k\}}=N_{\{n,n-k+1\}}$. 
 
 \begin{thm} The Lucas-Eulerian numbers are palindromic.
 \end{thm}
 
 \begin{proof} This is an immediate consequence of the bijection $k\mapsto n-k=k'$, which amounts to write
 \begin{equation*}
 E_{\{n,n-k\}} = \{n-k+1\} E_{\{n-1,n-k\}} + \{n-n+k+1\} E_{\{n-1,n-k-1\}}
 \end{equation*}
 as
 \begin{equation*}
 E_{\{n,k'\}} = \{k'+1\} E_{\{n-1,k'\}} + \{n-k'+1\} E_{\{n-1,k'-1\}}\,.
 \end{equation*}
 \end{proof}
 
Unlike Binomial and Narayana numbers, whose Lucas analogues have straightforward generalizations of their explicit closed formulas, there is not an equivalent formula known for the classical Eulerian numbers; that is, in terms of a quotient of products of nonnegative integers. There is a formula in terms of alternating sums though, but unfortunately the natural extension of the Lucas recurrence to negative numbers does not give a Lucas analogue of such formula that be polynomial in $s$ and $t$ with nonnegative integer coefficients. We elaborate more on this observation in Section~\ref{se:3}. 
 
Using Definition~\ref{de:LuEul} repeatedly we get an interesting formula for $E_{\{n,1\}}$. We have
 \begin{equation}\label{eq:LuEul-1stcolumn}
 E_{\{n,1\}} = \{2\} E_{\{n-1,1\}} + \{n\} = \sum_{j=0}^n \{n-j\} \{2\}^j\,.
 \end{equation}
 Expanding out the sum in Formula~\eqref{eq:LuEul-1stcolumn}, and replacing  $\{2\}$ with $x$, we get the polynomial 
 \begin{equation*}
 \big\{A_{n,1}\big\}(x) = \{n\} x^0 + \{n-1\} x^{1} + \cdots + \{1\} x^{n-1} + \{0\} x^n \,,
 \end{equation*}
 which is the Lucas version of the polynomial $A_{n,1}(x)=x^nP_{n,1}\left(\frac{1}{x}\right)$, where 
 \begin{equation*}
 P_{n,1}(x) = \left(x\frac{d}{dx}\right)\left(\frac{1-x^{n+1}}{1-x}\right) = x+2x^2+\cdots+nx^n\,.
 \end{equation*} 
 Therefore, we have $E_{\{n,1\}}=\big\{A_{n,1}\big\}(\{2\})$. Note that when $n+1$ is prime, the polynomials 
 \begin{equation*}
 \varphi_{n}(x)=1+x+x^2+\cdots+x^n = \frac{1-x^{n+1}}{1-x}\,,
 \end{equation*}
 are the cyclotomic polynomials $\Phi_{n+1}(x)$ used in the definition of Lucas atoms~\cite{ST_20}. They provide a convenient factorization of $\{n\}$. Recall also that there is a well-known formula that generates the classical Eulerian polynomials, here denoted by $E_n(x)$, in terms of the operator $x\frac{d}{dx}$ and the geometric series $\frac{1}{1-x}$; namely
 \begin{equation*}
\left(x\frac{d}{dx}\right)\left(\frac{1}{1-x}\right) = \frac{E_n(x)}{(1-x)^{n+1}}\,.
 \end{equation*}
 It is then reasonable to think that a general formula for $E_{\{n,k\}}$ may be given by the Lucas version of a polynomial obtained by applying the operator $x\frac{d}{dx}$ to a suitable transformation of $\varphi_n(x)$.
 
 \medskip
   
 \noindent\textbf{Conjecture.} The Lucas-Eulerian numbers are given by $E_{\{n,k\}}=\big\{A_{n,k}\big\}(\{k+1\})$, where $\big\{A_{n,k}\big\}$ is the Lucas version of the polynomial $A_{n,k}(x)$ defined as
 \begin{equation*}
 A_{n,k}(x) =  \left(x\frac{d}{dx}\right) \Big(T(\varphi_n(x))\Big)\,,
 \end{equation*}
 where $T$ is a suitable transformation of $\varphi_n(x)$.

 \section{On an alternating sum formula for Eulerian numbers}
 \label{se:3}
 
 For any integers $0\le k< n$, setting $E_{n,n}=1$ for all $n\ge 0$ and $E_{n,k}=0$ for $n<k$, a classical formula that holds for Eulerian numbers is
 \begin{equation}\label{eq:Eul-sum}
 E_{n,k} = \sum_{j=0}^{k+1} (-1)^{k+1-j} \binom{n+2}{k+1-j} j^{n+1} \,.
 \end{equation}
 A natural question that comes to mind is what the Lucas analogue of this formula should be. An immediate answer is to replace binomial coefficients and powers of nonnegative integers with their Lucas versions. But what is the Lucas version of $-1$? 
 
 \begin{table}[h]
\footnotesize
\begin{minipage}{0.45\textwidth}
\centering
\ra{1.2}
\begin{tabular}{c c}
\toprule
 $\bm{n\backslash k}$ & ${\color{red}1}$ \\
\midrule
\color{red}0 & 0  \\
\color{red} 1& 1 \\
\color{red}2 & $s^3+s^3t++2st^2$ \\
\color{red}3 & $s^4+s^4t+3s^2t^2+t^3$  \\
\color{red}4 & $s^5+s^5t+4s^3t^2+3st^3$  \\
\color{red}5 & $s^6+s^6t+5s^4t^2+6s^2t^3+t^4$  \\
\bottomrule
\end{tabular}
\caption{$E'_{\{n,1\}}$}
\label{table:2}
\end{minipage}%
\begin{minipage}{0.45\textwidth}
\centering
\ra{1.2}
\begin{tabular}{c c}
\toprule
 $\bm{n\backslash k}$ & ${\color{red}1}$ \\
\midrule
\color{red}0 & 0  \\
\color{red} 1& 1 \\
\color{red}2 & $2s+s^3+\frac{s^3}{t}$ \\
\color{red}3 & $3s^2+t+s^4+\frac{s^4}{t}$  \\
\color{red}4 & $4s^3+3st+s^5+\frac{s^5}{t}$  \\
\color{red}5 & $5s^4+6s^2t+t^2+s^6+\frac{s^6}{t}$  \\
\bottomrule
\end{tabular}
\caption{$E''_{\{n,1\}}$}
\label{table:3}
\end{minipage}
\end{table}
 
\noindent We have seen in Section~\ref{se:1} that for convenient recursive formulas that hold for Binomial and Narayana numbers, using Lucasnomials and replacing $-1$ with $t$ was enough to get consistent formulas for their Lucas analogues. Hence, let us do the same here. Define the Lucas number
 \begin{equation*}
 E'_{\{n,k\}} = \sum_{j=0}^{k+1} t^{k+1-j} \begin{Bmatrix}n+2 \\ k+1-j\end{Bmatrix} \{j\}^{n+1} \,.
 \end{equation*}
 We can easily check that $E_{\{n,k\}}\neq E'_{\{n,k\}}$ (compare Table~\ref{table:1} with Table~\ref{table:2}). Another plausible Lucas version of formula~\eqref{eq:Eul-sum} is obtained by extending the recurrence~\eqref{eq:Lucas-seq} to negative integers. In particular, we have $\{-1\}=\frac{1}{t}$ so that we get the Lucas number
\begin{equation*}
E''_{\{n,k\}} = \sum_{j=0}^{k+1} \left(\tfrac{1}{t}\right)^{k+1-j} \begin{Bmatrix}n+2 \\ k+1-j\end{Bmatrix} \{j\}^{n+1} \,.
 \end{equation*}
Once more, we have $E''_{\{n,k\}}\neq E_{\{n,k\}}$ (compare Table~\ref{table:1} with Table~\ref{table:3}). This time though, we see that $E''_{\{n,1\}}$ and $E_{\{n,1\}}$ are almost equal except for an extra term $s^{n+1}+\frac{s^{n+1}}{t}$ for $n\ge 2$.

\noindent We wonder what the Lucas version of the alternating sum formula~\eqref{eq:Eul-sum} is that be compatible with our Lucas analogue of Eulerian numbers.

 \section{Further comments}
 \label{se:4}

We expect to find an uncomplicated formula for Lucas-Eulerian numbers in general by using induction. It is not easy to obtain a formula for $E_{\{n,k\}}$ similar to the one given for $E_{\{n,1\}}$, as stated in the conjecture. This difficulty is due in part to the fact that Lucas generalized numbers are neither additive nor multiplicative, for instance $\{2\}+\{3\}\neq\{5\}$ and $\{2\}\{3\}\neq\{6\}$. Moreover, Lucas-Eulerian numbers are not always factorized in terms of Lucas atoms (we refer to \cite{ST_20} for definitions and notations). Try for example to factorize $E_{\{3,1\}}=3s^2+t$. Nevertheless, we can write $E_{\{3,1\}}$ as a linear combination of Lucas atoms  as follows
\begin{equation*}
3s^2+t = 4(s^2+t)-(s^2+3t)=4P_3(s,t) - P_6(s,t)\,.
\end{equation*}

Despite the elementary way of defining Lucas-Eulerian numbers, we have been unable to find a combinatorial interpretation for them. We hope that the reader will be tempted to find such an interpretation. 

The recursive approach used here to define Lucas-Eulerian numbers can be used to define Lucas versions of a vast collection of combinatorial constants for which recursive relations are known. For instance, the Lucas analogue of the Stirling numbers $St2_{n,k}$ of the second kind are given by
\begin{equation}\label{eq:LuSt2}
St2_{\{n,k\}} = \{k\} St2_{\{n-1,k\}} + St2_{\{n-1,k-1\}}\,,
\end{equation}
with the usual initial conditions $St2_{\{0,0\}}=1$, $St2_{\{n,k\}}=1$ for $n=k$ or $k=1$, $St2_{\{j,0\}}=0$ for $j>0$, and $St2_{\{n,k\}}=0$ for $n<k$. 

\begin{table}[h!]
\tiny
\begin{minipage}{\textwidth}
\centering
\ra{1.2}
\begin{tabular}{c c c c c c p{0.8cm} p{0.8cm}}
\toprule
 $\bm{n\backslash k}$ & ${\color{red}0}$ & ${\color{red}1}$ & ${\color{red}2}$ & ${\color{red}3}$ & ${\color{red}4}$ & ${\color{red}5}$  \\
\midrule
\color{red}0 & 1 & 0 & 0 & 0 & 0 & 0 \\
\color{red} 1 & 0 & 1 & 0 & 0 & 0 & 0 \\
\color{red}2 & 0 & $1$ & $1+s$ & 0 & 0 & 0 \\
\color{red}3 & 0 & $1$ & $1+s+s^2$ & 1 & 0 & 0 \\
\color{red}4 & 0 & $1$ & $1+s+s^2+s^3$ & $1+s+s^2+t$ & 1 & 0 \\
\color{red}5 & 0 & $1$ & $1+s+s^2+s^3+s^4$ & $1+s+2s^2++s^3+s^4+t+st+2s^2t+t^2$ & $1+s+s^2+s^3+t+2st$ & 1 \\
\bottomrule
\end{tabular}
\caption{A sample of Lucas-Stirling numbers of the second kind}
\label{table:4}
\end{minipage}
\end{table}

\noindent It follows from~\eqref{eq:LuSt2} that the Lucas-Stirling numbers of the second kind are polynomials in $s$ and $t$ with nonnegative integer coefficients. Table~\ref{table:4} shows some of them. We are not aware of any work on these numbers in the literature.

\noindent Here is another enlightening example; for $n\ge 2$ and with initial conditions $M_0=M_1=1$, the standard Motzkin numbers $M_n$ are known to satisfy the following recursive formula, 
\begin{equation}\label{eq:Mo}
M_{n} = \frac{2n+1}{n+2} M_{n-1} + \frac{3n-3}{n+2} M_{n-2}\,.
\end{equation}

\noindent The natural Lucas analogue of relation~\eqref{eq:Mo} is

\begin{equation}\label{eq:LuMo}
M_{\{n\}} = \frac{\{2n+1\}}{\{n+2\}} M_{\{n-1\}} + \frac{\{3n-3\}}{\{n+2\}} M_{\{n-2\}}\,,
\end{equation}
with initial conditions $M_{\{0\}}=M_{\{1\}}=1$. Formula~\eqref{eq:LuMo} gives rational expressions in $s$ and $t$ for $M_{\{n\}}$ in general. For instance, 
\begin{equation*}
M_{\{2\}} = \frac{s^4+3s^2t+t^2+s^2+t}{s^3+2st}\,.
\end{equation*}
On the other hand, Motzkin numbers also satisfy a formula in terms of binomials coefficients and Catalan numbers, whose straightforward Lucas analogue is
\begin{equation}\label{eq:LuMoBinCat}
M_{\{n\}} = \sum_{k=0}^{\lfloor n/2\rfloor} \begin{Bmatrix}n\\2k\end{Bmatrix} C_{\{k\}}\,,
\end{equation}
where $C_{\{k\}} = {\tifrac{1}{k+1}}{\tiny\begin{Bmatrix}2k \\ k\end{Bmatrix}}$. Since Lucasnomials and Lucas-Catalan numbers are polynomials in $s$ and $t$ with nonnegative integer coefficients~\cite{Bennett_et_al_20}, so are the Lucas-Motzkin numbers by Formula~\eqref{eq:LuMoBinCat}.  This apparent contradiction with \eqref{eq:LuMo} is due to the fact that the recursive formula used to define the Lucas numbers $M_{\{n\}}$ is not the right one. In the same way we did in Section~\ref{se:1}, by using a recursive relation different from the common one that holds for binomial numbers, here we have to pick a different recursive relation whose Lucas analogue be compatible with Formula~\eqref{eq:LuMoBinCat}.

In ongoing work, we elaborate more on this observation~\cite{Agapito_Rio23}.

\end{document}